\documentclass[11pt]{article}
\usepackage{graphicx,epic,eepic}
\usepackage{amsmath, amsthm, verbatim, amsfonts, amssymb, color}
\usepackage{mathrsfs}
\usepackage{dsfont}
\usepackage[a4paper,margin=3cm]{geometry}
\usepackage{epstopdf}

\numberwithin{equation}{section}
\renewcommand\labelenumi{\textup{\alph{enumi})}}
\renewcommand\theenumi\labelenumi

\newcommand\Pp{\mathds{P}}
\newcommand\nat{\mathds{N}}
\newcommand\real{{\mathds{R}}}
\newcommand\I{\mathds{1}}

\newcommand\dup{\mathrm{d}}
\newcommand\eup{\mathrm{e}}

\usepackage{marginnote}

\newcommand{\wei}{\color{green}}

\usepackage{amssymb,amsmath,amsthm,mathdesign}
\usepackage{pdfsync}
\usepackage{color}
\usepackage[colorlinks]{hyperref}

\newtheorem{theorem}{Theorem}[section]
\newtheorem{lemma}[theorem]{Lemma}
\newtheorem{definition}[theorem]{Definition}
\newtheorem{proposition}[theorem]{Proposition}

\newtheorem{assumption}[theorem]{Assumption}
\newtheorem{remark}[theorem]{Remark}
\newtheorem{example}[theorem]{Example}
\numberwithin{equation}{section}

\newcommand{\be}{\begin{equation}}
\newcommand{\ee}{\end{equation}}

\newcommand{\bes}{\begin{equation*}}
\newcommand{\ees}{\end{equation*}}

\def\E{\bE}

\def\cG{\mathcal{G}}

\def\bE{\mathbb{E}}

\renewcommand{\d}{{\rm d}}

\renewcommand{\geq}{\geqslant}
\renewcommand{\leq}{\leqslant}

\def\m1{\mathbf{1}}

\allowdisplaybreaks[1]

\allowdisplaybreaks[1]

\title{\textbf{Finite Time Blowup of Solutions to SPDEs with Bernstein Functions of the Laplacian}}
\author{	\textbf{Chang-Song Deng} \\School of Mathematics and Statistics, Wuhan University, Wuhan 430072, China\\
Email: dengcs@whu.edu.cn
\and \textbf{Wei Liu} \\ Department of Mathematics, Shanghai Normal University, Shanghai 200234, China\\
Email: weiliu@shnu.edu.cn
\and \textbf{Erkan Nane}\\ Department of  Mathematics and Statistics, Auburn University, Alabama 36849, USA \\
Email: ezn0001@auburn.edu
}
\date{}

\begin{document}
\maketitle

\begin{abstract}
The blowup in finite time of solutions to SPDEs
\begin{equation*}
\partial_tu_t(x)=-\phi(-\Delta)u_t(x)
    +\sigma(u_t(x))\dot{\xi}(t,x),
    \quad t>0,x\in\real^d,
\end{equation*}
{\color{blue} is} investigated, where $\dot{\xi}$ could be either a white noise or a colored noise and $\phi:(0,\infty)\to (0,\infty)$ is a Bernstein function. The sufficient conditions on $\sigma$, $\dot{\xi}$ and the initial value that imply the non-existence of the global solution are discussed. The results in this paper generalise those in \cite{FLN2019}, where the fractional Laplacian case was considered, i.e. $\phi(-\Delta)=(-\Delta)^{\alpha/2}$ ($1<\alpha<2$).
\end{abstract}

\maketitle

\noindent

\section{Introduction} \label{Sec:Intro}

The finite time blowup or the non-existence of the global solutions for stochastic equations has been receiving interests of many scholars. For stochastic ordinary differential equations, a sufficient condition for almost sure explosion for one-dimensional equations was given by Feller in \cite{Feller1954}. The result was generalised to the multi-dimensional case by Khasminskii in \cite{Khasminskii1960}. A different approach based on the Lyapunov function to study the explosion of the solutions was discussed by Chow and Khasminskii in \cite{CK2011,CK2014}. When the driving noise is a L\'evy process, Xing and Li in \cite{XL2017} obtained some results about the explosive solutions of stochastic differential equations.

When the equations under investigation are stochastic partial differential equations, Mueller in \cite{Mueller2000} and Mueller and Sowers in \cite{MS1993} studied stochastic heat equations driven by the space-time white noise with Dirichlet boundary condition and revealed the critical value for the blowup in finite time in the pathwise sense. A series of papers by Chow \cite{Chow2009,Chow2011,Chow2012} obtained results of explosive solutions for different types of stochastic partial differential equations. Bonder and Groisman in \cite{FG2009} took a drift term into consideration and discussed the sufficient condition on the drift term that leads to the almost surely finite time blowup of solutions to stochastic reaction-diffusion equations with the space-time white noise on the spatial domain $(0,1)$. Lv and Duan in \cite{LD2015} studied similar equations with the higher dimension of the spatial domain in the moment sense and found the interplay between the drift and diffusion for the explosion. When the driving  noise is a jump process, Bao and Yuan in \cite{BY2016} investigated the blowup in the $L^p$ sense for stochastic reaction-diffusion equations. Li, Peng and Jia extended the results to stochastic partial differential equations driven by a class of L\'evy processes in \cite{LPJ2017}. When the delay effects are taken into consideration, Chow and Liu in \cite{CL2012} discussed the explosive solutions to a class stochastic functional parabolic equations of retarded type. Lv, Wang and Wang studied a class of stochastic delayed evolution equations and showed that the delay term can induce the explosion in \cite{LWW2016}. Li in \cite{Li2018} generalised the result to the case of the L\'evy noise.

More recently, stochastic partial differential equations with fractional operators have been received increasingly attentions. Foondun and Parshad studied the non-existence of global random field solutions and finite energy solutions to stochastic heat equations with the fractional Laplacian driven by white noise in \cite{FP2015}. Wang in \cite{Wang2019} investigated similar problems under different conditions. Foondun, Liu and Nane obtained some non-existence results for fractional stochastic heat equations driven by colored noise on the multi-dimensional spatial domain in \cite{FLN2019}. Asogwa, Mijena and Nane in \cite{AMN2019} considered the blowup of solutions to more general fractional stochastic partial differential equations, where the temporal operator was also fractional.

In this paper, we investigate the non-existence of global solutions to SPDEs of the following form
\begin{equation}\label{eqSec1}
\partial_tu_t(x)=-\phi(-\Delta)u_t(x)
    +\sigma(u_t(x))\dot{\xi}(t,x),
    \quad t>0,x\in\real^d,
\end{equation}
where $\dot{\xi}$ could be either a white noise or a colored noise, $\phi:(0,\infty)\to (0,\infty)$ is a Bernstein function, i.e.\ a $C^\infty$-function such that $\phi\geq 0$ and with alternating derivatives $(-1)^{n} \phi^{(n)}\leq 0$,
$n\in\nat$, and $\sigma:\real^d\rightarrow[0,\infty)$
satisfies
\begin{assumption}\label{Assumption 1} $\sigma$ is locally Lipschitz satisfying
the following growth condition: there exists some
$\gamma>0$ such that
$$
    \sigma(x)\geq|x|^{1+\gamma}\quad
    \text{for all $x\in\real^d$}.
$$
\end{assumption}

It is clear that equation \eqref{eqSec1} recovers those equations discussed in \cite{FLN2019} when $\phi(s)=s^{\alpha/2}$ with $1<\alpha<2$. In addition, by choosing
different Bernstein functions, several other different types of SPDEs can be covered by  equation \eqref{eqSec1}.

It is well known, see e.g.\ \cite[Theorem 3.2]{SSV},
that every Bernstein function enjoys a unique
L\'{e}vy--Khintchine representation
\begin{equation}\label{bern}
    \phi(s)=\phi(0+)+bs+\int_{(0,\infty)}
    \left(1-\eup^{-sx}\right)\,\nu(\dup x),\quad s>0,
\end{equation}
where $\phi(0+)\geq0$ is the killing term, $b\geq0$ is the
drift parameter, and $\nu$ is a L\'{e}vy measure,
that is, a Radon measure on $(0,\infty)$ satisfying
$\int_{(0,\infty)}
    \left(1\wedge x\right)\,\nu(\dup x)<\infty$.
Let $S_t$ be a subordinator (without killing)
whose characteristic (Laplace)
exponent is the Bernstein function $\phi$
with $\phi(0+)=0$; it is a non-decreasing L\'{e}vy
process on $[0,\infty)$ with $S_0=0$, and its
Laplace transform is of the form
$$
    \E\,\eup^{-rS_t}=\eup^{-t\phi(r)},\quad r>0,t\geq0.
$$

We will use the following assumptions on the
Bernstein function $\phi$.

\begin{assumption}\label{Assumption 2} The Bernstein function $\phi$ satisfies
$$
    \int_{0^+}\frac{\dup s}{s^\delta\phi(s)}=\infty
    \quad\text{for some $\delta<1/2$}.
$$
\end{assumption}

\begin{assumption}\label{Assumption 3} The Bernstein function $\phi$ satisfies
    $$
        \liminf_{s\rightarrow\infty}\frac{\phi(s)}{\log s}
        >0.
    $$
\end{assumption}

\begin{assumption}\label{Assumption grow} The Bernstein function $\phi$ satisfies
    $$
        \liminf_{r\to\infty}
        \frac{\phi(\lambda r)}{\phi(r)}>1
        \quad \text{for some \textup{(}hence,
        all\textup{)} $\lambda>1$}.
    $$
\end{assumption}

\begin{remark}
    Since Assumption \ref{Assumption grow} implies
    that $\phi$ grows at least like a \textup{(}fractional\textup{)}
    power \textup{(}cf. \cite[Lemma 2.3\,(i)]{DSS17}\textup{)}, Assumption \ref{Assumption grow}
    implies Assumption \ref{Assumption 3}
\end{remark}

\begin{example}
We list here some examples of  $\phi$ that
satisfy the above assumptions.
We refer the reader to \cite{SSV} for more examples of such
Bernstein functions.

\begin{enumerate}
\item   \textup{(}\textbf{Stable subordinators}\textup{)} Let $\phi(s)=s^{\alpha/2}$, $0<\alpha<2$.
    Then Assumption \ref{Assumption 2} holds iff
    $\alpha>1$, and Assumptions
    \ref{Assumption 3}-\ref{Assumption grow} always hold.

\item   \textup{(}\textbf{Relativistic stable subordinators}\textup{)} Let $\phi(s)
    =(s+m^{2/\alpha})^{\alpha/2}-m$, $0<\alpha<2$, $m>0$.
    Then Assumptions \ref{Assumption 2}-\ref{Assumption grow} always hold.

\item   \textup{(}\textbf{Gamma subordinators}\textup{)} Let $\phi(s)
    =\log(1+s)$.
    Then Assumptions \ref{Assumption 2}-\ref{Assumption 3} hold, but
    Assumption \ref{Assumption grow} does not hold.

\item   \textup{(}\textbf{Geometric stable subordinators}\textup{)} Let $\phi(s)
    =\log(1+s^{\alpha/2})$, $0<\alpha<2$.
    Then Assumption \ref{Assumption 2} hold
    iff $\alpha>1$, Assumption \ref{Assumption 3}
    always hold, and Assumption \ref{Assumption grow} does not hold.

\item   Let $\phi(s)
    =\log\left(1+s+\sqrt{(1+s)^2-1}\right)$.
    Then Assumption \ref{Assumption 3} holds, but
    Assumptions \ref{Assumption 2} and
    \ref{Assumption grow} do not hold.

\item   Let $\phi(s)
    =\log^2\left(1+s+\sqrt{(1+s)^2-1}\right)$.
    Then Assumptions \ref{Assumption 2}-\ref{Assumption grow} hold.

\item
    Let $\phi(s)=s^{\alpha/2}\log^{\beta/2}(1+s)$,
    $0<\alpha<2$, $0<\beta\leq2-\alpha$.
    Then Assumption \ref{Assumption 2} holds iff $\alpha+\beta>1$, and Assumptions
    \ref{Assumption 3}-\ref{Assumption grow} always hold.

\item
    Let $\phi(s)=s^{\alpha/2}\log^{-\beta/2}(1+s)$,
    $0<\beta\leq\alpha<2$.
    Then Assumption \ref{Assumption 2} holds iff $\alpha-\beta>1$, and Assumptions
    \ref{Assumption 3}-\ref{Assumption grow} always hold.

\item
    Let $\phi(s)=s(1+s)^{-\alpha/2}$,
    $0<\alpha<2$.
    Then Assumptions \ref{Assumption 2}-\ref{Assumption grow} always hold.
    \end{enumerate}
\end{example}

With suitable requirements imposed on $\phi$, we discuss the finite time blowup of the solutions of  \eqref{eqSec1} in $L^2$ sense where  the driving  noise is  white or colored in space. We also investigate the influence of the initial values. For such a general setting with the Bernstein function, new ideas and techniques are needed for proofs of reuslts in our paper.
Other related results
concerning Bernstein functions of the Laplacian
can be found in \cite{HIL12, HL12, KM18}.

This paper is organized  as follows. The main results are presented in Section 2, where the case of white noise goes to Subsection 2.1 and the case of colored noise is discussed in Subsection 2.2. In order to prove our results, we make some preparations in Section 3. The proofs of main theorems are given  in Section 4. Section 5 concludes this paper and discusses some possible future research directions.

Throughout the paper, we denote by $B(\zeta, R)$ the open
Euclidean ball of radius $R>0$ centered
at $\zeta\in\real^d$.

\section{Main results}

Main theorems of this paper are stated in this section which is divided into two parts for cases of the white noise and the space colored noise. In each part, we discuss two different conditions on the initial values.

\subsection{Case I: white noise} \label{SubSecWN}

Consider
\begin{equation}\label{eq:white}
    \partial_tu_t(x)=-\phi(-\Delta)u_t(x)
    +\sigma(u_t(x))\dot{W}(t,x),
    \quad t>0,x\in\real,
\end{equation}
where $\dot{W}(t,x)$ is a white noise, and
$\phi:(0,\infty)\to (0,\infty)$ is a
Bernstein function.
\par
A mild solution to \eqref{eq:white} in the sense of Walsh \cite{walsh} is given  by
\begin{equation}\label{mild}
u_t(x)=
(\cG u)_t(x)+ \int_{\real^d}\int_0^t p_{t-s}(x-y)\sigma(u_s(y))\,W(\d s\,\d y),
\end{equation}
where
\begin{equation}\label{deter}
(\cG u)_t(x):=\int_{\real^d} p_t(x-y)u_0(y)\,\d y,
\end{equation}
and $p_t(x-y)=p_t(x,y)$ denotes the heat kernel of $-\phi(-\Delta)$. If we further have
\begin{equation}\label{moments}
\sup_{t\in[0,\,T]} \E|u_t(x)|^k<\infty \quad\text{for all
$T>0$, {\color{blue}$k\geq2$}, and $x\in\real$}
%\footnote{{\color{blue}Please double check: for all $x\in\real$,
%or take supremum (i.e. $\sup_{x\in\real^d}\sup_{t\in[0,T]}...<\infty$)?}{ \wei I think for all $x\in\real$ is fine here.}},
\end{equation}
then we say that $u$ is a {\it random field} solution to equation \eqref{eq:white}.
A sufficient condition of the existence of solutions
when $\sigma $ is globally Lipschitz is the Dalangs
condition (see \cite{foondun-davar-09}) that boils down to the following in our case
\begin{equation}
\int_{\real}\frac{1}{1+\phi(|\xi|^2)}\,\dup\xi<\infty.
\end{equation}

Since we investigate the finite time blowup, the global existence of the solution is not expected. By using the technique of  stopping times, it is clear to see the existence and uniqueness of a local solution \cite{Davar,walsh} given Assumption \ref{Assumption 1}. So from now on, by saying $u_t(x)$ is the solution to equations investigated in this paper, we mean the local solution up to a stopping time.

Throughout this paper, the initial condition $u_0$ will always be a non-negative bounded deterministic function.

We first  assume that the initial condition is bounded below by a positive constant.

\begin{theorem}\label{theorem-white}
    Let $d=1$ and $u_t$ be be the solution to \eqref{eq:white}. Suppose that  the Assumptions \ref{Assumption 1}  and \ref{Assumption 2} hold, and that
    $\kappa:=\inf_{x\in\real^d}u_0(x)>0$. Then there
    exists a $t_0>1$ such that for all
    $t\geq t_0$ and $x\in\real$,
    $$
        \E|u_t(x)|^2=\infty.
    $$
\end{theorem}

\begin{remark}
The above result states that provided that the initial function is bounded below, the second moment will eventually cease to be finite  for equations driven
by \textup{(}space-time\textup{)} white noise
\end{remark}
Next, we remove the assumption that the initial condition is
bounded below by a constant, i.e. $\kappa > 0$, and
impose the following requirement on the initial values instead.

Let
\begin{equation}\label{eqK0}
K_{u_0}:=\int_{B(0,\,1)}u_0(x)\,\d x.
\end{equation}
We have taken $B(0,\,1)$ as a matter of convenience.

\begin{theorem}\label{theorem2-white}
   Suppose that $u_t$ is a solution of \eqref{eq:white},and  Assumptions \ref{Assumption 1} and \ref{Assumption 3} hold. Then there exists
    $t_0>0$ and $K>0$ such that for
    all $t\geq t_0$ and $x\in\real^d$
    $$
        \E|u_t(x)|^2=\infty\quad
        \text{whenever $K_{u_0}\geq K$}.
    $$
\end{theorem}

\subsection{Case II: colored noise} \label{SubSecCdN}
In this section, we study equations driven by  space colored noise {\color{blue} $\dot F(t,x)$}. Consider
\begin{equation}\label{eq:colored}
\partial_t u_t(x)=
 -\phi(-\Delta)u_t(x)+ \sigma (u_t(x)) \dot F(t,x),
 \quad t>0,x\in \real^d.
\end{equation}
The corresponding mild solution in the sense of Walsh \cite{walsh} is given by
\begin{equation}\label{mildcod}
u_t(x)=
(\cG u)_t(x)+ \int_{\real^d}\int_0^t p_{t-s}(x-y)\sigma(u_s(y))\,F(\d s\,\d y).
\end{equation}
Here, we are interested in the random field solution as well. But we need to impose some conditions on the noise term:
\begin{align*}
\E[\dot F(s,x)\dot F(t,y)]=\delta_0(t-s)f(x,\,y),
\end{align*}
where ${\color{blue} 0 <} f(x,\,y)\leq g(x-y)$ and $g$ is a locally integrable function on $\real^d$ with a possible singularity at $0$ and  satisfies
\begin{equation*}
\int_{\real^d}\frac{\hat{g}(\xi)}{1+\phi(|\xi|^2)}\,
\d \xi<\infty,
\end{equation*}
where $\hat{g}$ denotes the Fourier transform of $g$.
Since the $\sigma$ in this paper is allowed to grow polynomially, the equations discussed here only have unique local random field solutions as we have already mentioned at the beginning of Subsection \ref{SubSecWN}.
\par

For equations with colored noise, our results need more conditions on the spatial correlation of the noise.
\begin{assumption}\label{color}
There exists $R>0$ such that
\begin{equation*}
K_{R,f}:=\inf_{x,\,y\in B(0,\,R)}f(x,\,y)>0.
\end{equation*}
\end{assumption}

This assumption is quite mild. It is not hard to see that all the following  examples of $f(x,\,y)$ satisfy
Assumption \ref{color}.
\begin{itemize}
\item Riesz kernel: \begin{equation*}f(x,\,y)=g(x-y)=\frac{1}{|x-y|^\beta}.
    \end{equation*}
    Here we require $0<\beta<d$ and, since $\hat{g}(\xi)=C\,|\xi|^{\beta-d}$ for
    some constant $C=C(\beta,d)>0$, we also require that
  \begin{equation}\label{riesz-bound-dalang-cond}
    \int_0^\infty\frac{r^{\beta-1}}{1+\phi(r^2)}\,
    \dup r<\infty.
   \end{equation}
    When $\phi(s)=s^{\alpha/2}$ with $\alpha\in(0,2)$,
    the inequality \eqref{riesz-bound-dalang-cond} holds
    if $0<\beta< d \wedge \alpha.$

\item The exponential-type kernel: $f(x,\,y)=\exp[-(x\cdot y)].$

\item The Ornstein-Uhlenbeck-type kernel: $f(x,\,y)=\exp[-|x-y|^\alpha]$ with $\alpha\in(0,\,2].$

\item Poisson kernel: $$f(x,\,y)=g(x-y)=\left(\frac{1}{|x-y|^2+1}\right)^{(d+1)/2}.$$
In this case, since $\hat{g}(\xi)=C\,\eup^{-|\xi|}$ for
some constant $C=C(d)>0$, the Dalang condition boils down to the requirement that
$$
\int_{0}^\infty \frac{r^{d-1}\eup^{-r}}{1+\phi(r^2)}\,
\dup r<\infty.
$$

\item Cauchy kernel: $$f(x,\,y)=\sum_{j=1}^d\left(\frac{1}{1+(x_j-y_j)^2}\right).$$
\end{itemize}

\begin{theorem}\label{theo2}
Let $u_t$ be the solution to \eqref{eq:colored} and suppose that Assumption \ref{color} holds. For any $t_0>0$,
there exists a positive number
$\kappa_0=\kappa_0(t_0)$ such that
for all $t\geq t_0$ and $ x \in \real^d$,
\begin{equation*}
\E|u_t(x)|^2=\infty, \quad \text{whenever
$\kappa:=\inf_{x\in\real^d}u_0(x)\geq \kappa_0$}.
\end{equation*}
\end{theorem}

Similar to the case of white noise in Section \ref{SubSecWN}, we remove the assumption that the initial condition is bounded below by a constant now.

{%\color{red}
\begin{theorem}\label{energy-colored}
Let $u_t$ be the solution to \eqref{eq:colored}. Then, under Assumptions \ref{color}, there exists a $t_0\geq 0$ such that for all $t\geq t_0$ and $x\in \real^d$,
\begin{equation*}
\E|u_t(x)|^2=\infty\quad\text{whenever}\quad K_{u_0}\geq K,
\end{equation*}
where $K_{u_0}$ has the same format as that in \eqref{eqK0} with $B(0,\,1)\subset \real^d $ in this case and $K$ is a positive constant.
\end{theorem}

So far, all the equations discussed in Subsections \ref{SubSecWN} and \ref{SubSecCdN} are in the whole spatial domain. As mentioned in Section \ref{Sec:Intro}, many previous works investigated equations with some boundary conditions.
For fixed $R>0$, the final theorem of our paper investigates the equation in the ball $B(0,\,R)$ with Dirichlet boundary conditions.

\begin{theorem}\label{Dirichlet}
Assume that $\phi$ is given by \eqref{bern}
    with $\phi(0+)=b=0$ and satisfies
Assumption \ref{Assumption grow},
fix $R>0$ and consider
\begin{equation}\label{eq:dir}
\partial_t u_t(x)=
 \mathcal{L}u_t(x)+ \sigma (u_t(x))\dot F(t,x)\quad{t>0}\quad\text{and}\quad x\in B(0,\,R).
 \end{equation}
Here $\mathcal{L}$ is the generator of a L\'evy  process corresponding to $-\phi(-\Delta)$ killed upon exiting the ball $B(0,\,R)$. The noise $\dot F$ is taken to be spatially colored with correlation function satisfying all the conditions stated above. For any $\epsilon>0$, there exist $t_0>0$ and $\kappa_0>0$, such that if
$\inf_{x\in B(0,\,R/2)}u_0(x)>\kappa_0$, then
\begin{equation*}
\E|u_t(x)|^2=\infty\quad\text{for all}\quad t\geq t_0 \quad\text{and}\quad x\in B(0,\,R-\epsilon).
\end{equation*}
\end{theorem}

\section{Some preparations} \label{Sec:MathPre}

\subsection{Berstein function and the corresponding subordinator}

The following definition is taken from
\cite[Definition 1.1 (1)]{kim-mimica-2018}.

\begin{definition}
    We say that $g:(0,\infty)\rightarrow(0,\infty)$
    satisfies the lower scaling condition
    if there exist $a\geq0$, $\gamma>0$
    and $C_L\in(0,1]$ such that
    $$
        \frac{g(\lambda r)}{g(r)}\geq C_L\lambda^\gamma
            \quad\text{for all $\lambda\geq1$ and $r>a$}.
    $$
\end{definition}

\begin{lemma}\label{equiva}
Let $g:(0,\infty)\to(0,\infty)$ be a non-decreasing function. Then the following statements are equivalent:
        \begin{enumerate}
        \item[\textbf{\textup{i)}}]
        $\displaystyle\liminf_{r\to\infty}\frac{g(\lambda_0 r)}{g(r)}>1$
        \ for some $\lambda_0>1$;

        \item[\textbf{\textup{ii)}}]
        $g$ satisfies the lower scaling condition.
        \end{enumerate}
\end{lemma}

\begin{proof}
    First, choosing $\lambda_0>1$ large enough such that
    $C_L\lambda_0^\gamma>1$, we get the
    direction ii) $\Rightarrow$ i). Conversely, suppose that i) holds true for some $\lambda_0>1$. For $\lambda\geq1$, let $k:=\lfloor\log_{\lambda_0}\lambda\rfloor+1$,
    where $\lfloor x\rfloor$ denotes the integer part of a non-negative real number $x\geq 0$. Since i) implies that there exist $c_1>1$ and $c_2>0$ such that
    $$
            g(\lambda_0r)\geq c_1g(r),\quad r> c_2,
    $$
    we find that for all $r> c_2$ and $\lambda\geq1$
    $$
            g(\lambda r)\geq g\left(\lambda_0^{k-1}r\right)
            \geq c_1^{k-1}g(r)>
            c_1^{\log_{\lambda_0}\lambda-1}g(r)
            =\frac{g(r)}{c_1}
            \lambda^{\log_{\lambda_0}c_1}.
    $$
    This means that $g$
    satisfies the lower scaling condition
    with $a=c_2$, $\gamma=\log_{\lambda_0}c_1$, and
    $C_L=c_1^{-1}$.
\end{proof}

Combining Lemma \ref{equiva}
    and \cite[Lemma 2.2\,(i)]{DSS17}, we get
    the following result.

\begin{lemma}\label{inverse}
Let $g:(0,\infty)\to(0,\infty)$ be a strictly
increasing function. Then the
following statements are equivalent:
    \begin{enumerate}
    \item[\textbf{\textup{i)}}]
        $\displaystyle\lim_{r\to\infty}g(r)=\infty$
        \ and \
        $\displaystyle\limsup_{r\to\infty}
        \frac{g^{-1}(\lambda_0r)}{g^{-1}(r)}<\infty$
        \ for some $\lambda_0>1$;

        \item[\textbf{\textup{ii)}}]
        $\displaystyle\lim_{r\to\infty}g(r)=\infty$
        \ and \
        $\displaystyle\limsup_{r\to\infty}\frac{g^{-1}(\lambda r)}{g^{-1}(r)}<\infty$
        \ for all $\lambda>1$;

        \item[\textbf{\textup{iii)}}]
            $\displaystyle\liminf_{r\to\infty}
            \frac{g(\lambda_0r)}{g(r)}>1$
            \ for some $\lambda_0>1$;

        \item[\textbf{\textup{iv)}}]
        $g$ satisfies the lower scaling condition;
        \end{enumerate}
    If $g$ is concave, then \textup{\bfseries i)--iv)} are also equivalent to:
        \begin{enumerate}
        \item[\textbf{\textup{v)}}]
        $\displaystyle\liminf_{r\to\infty}\frac{g(\lambda r)}{g(r)}>1$
        \ for all $\lambda>1$.
        \end{enumerate}
\end{lemma}

\begin{lemma}\label{moment}
    Suppose that Assumption \ref{Assumption 3} holds. Then for any $\beta>0$,
    $$
        \lim_{t\rightarrow\infty}S_t^{-\beta}=0.
    $$
\end{lemma}

\begin{proof}
    Using the identity
    $$
        x^{-\beta}=\frac{1}{\Gamma(\beta)}
        \int_0^\infty\eup^{-rx}r^{\beta-1}
        \,\dup r,\quad x\geq0,
    $$
    and Tonelli's theorem,
    $$
        \E S_t^{-\beta}=\frac{1}{\Gamma(\beta)}\,
        \E\left[\int_0^\infty\eup^{-rS_t}r^{\beta-1}
        \,\dup r\right]
        =\frac{1}{\Gamma(\beta)}
        \int_0^\infty\eup^{-t\phi(r)}
        r^{\beta-1}
        \,\dup r.
    $$
    By Assumption \ref{Assumption 3}, there exist $c_1>0$ and $c_2>1$
    such that
    $$
        \phi(r)\geq c_1\log r\quad\text{for all $r\geq c_2$}.
    $$
    This implies that for $t\geq2\beta/c_1$,
    $$
        \eup^{-t\phi(r)}
        r^{\beta-1}
        \leq\I_{\{r<c_2\}}r^{\beta-1}
        +\I_{\{r\geq c_2\}}r^{-c_1t+\beta-1}
        \leq\I_{\{r<c_2\}}r^{\beta-1}
        +\I_{\{r\geq c_2\}}r^{-\beta-1}.
    $$
    By the dominated convergence theorem,
    \[
        \lim_{t\rightarrow\infty}\E S_t^{-\beta}
        =\frac{1}{\Gamma(\beta)}
        \int_0^\infty\lim_{t\rightarrow\infty}
        \eup^{-t\phi(r)}
        r^{\beta-1}
        \,\dup r=0.  \qedhere
    \]
\end{proof}

%\section{Some useful lemmas}

The following lemma is essential due
to \cite[Proposition 2.4]{mimica}. We include a simple
proof for completeness.

\begin{lemma}\label{probab}
    For any $t>0$,
    $$
        \Pp\left(
        \left[\phi^{-1}\left(
        \frac{2}{t}\right)\right]^{-1}
        \leq S_t\leq
        \left[\phi^{-1}\left(
        \frac{1}{2t}\right)\right]^{-1}
        \right)
        \geq\frac{\eup^{3/2}-2\eup+1}{\eup(\eup-1)}
        >0.
    $$
\end{lemma}

\begin{proof}
    It follows from the Chebyshev inequality that
    \begin{align*}
        &\Pp\left(
        \left[\phi^{-1}\left(
        \frac{2}{t}\right)\right]^{-1}
        \leq S_t\leq
        \left[\phi^{-1}\left(
        \frac{1}{2t}\right)\right]^{-1}
        \right)\\
        &\quad\qquad=1-\Pp\left(
        S_t>\left[\phi^{-1}\left(
        \frac{1}{2t}\right)\right]^{-1}
        \right)
        -\Pp\left(
        S_t<\left[\phi^{-1}\left(
        \frac{2}{t}\right)\right]^{-1}
        \right)\\
        &\quad\qquad=1-\Pp\left(
        1-\eup^{-\phi^{-1}\left(
        \frac{1}{2t}\right)S_t}
        >1-\eup^{-1}
        \right)
        -\Pp\left(
        \eup^{-\phi^{-1}\left(
        \frac{2}{t}\right)S_t}>\eup^{-1}
        \right)\\
        &\quad\qquad\geq1-\frac{1}{1-\eup^{-1}}\,
        \E\left[
        1-\eup^{-\phi^{-1}\left(
        \frac{1}{2t}\right)S_t}
        \right]
        -\eup\cdot\E
        \eup^{-\phi^{-1}\left(
        \frac{2}{t}\right)S_t}\\
        &\quad\qquad=1-\frac{1-\eup^{-1/2}}{1-\eup^{-1}}
        -\eup\cdot\eup^{-2}\\
        &\quad\qquad=\frac{\eup^{3/2}-2\eup+1}
        {\eup(\eup-1)}. \qedhere
    \end{align*}
\end{proof}

Using the same argument, it is easy to get
the following estimate.

\begin{lemma}\label{subord}
    For any $t>0$ and $c\in(0,1)$,
    $$
        \Pp\left(
        S_t\leq\left[\phi^{-1}\left(
        \frac{c}{t}\right)\right]^{-1}
        \right)
        \geq\frac{\eup^{1-c}-1}{\eup-1}.
    $$
\end{lemma}

\begin{proof}
    By the Chebyshev inequality,
    \begin{align*}
        \Pp\left(
        S_t\leq\left[\phi^{-1}\left(
        \frac{c}{t}\right)\right]^{-1}
        \right)
        &=1-\Pp\left(1-
        \exp\left[-\phi^{-1}\left(
        \frac{c}{t}\right)S_t\right]>1-\eup^{-1}
        \right)\\
        &\geq
        1-\frac{1}{1-\eup^{-1}}\,
        \E\left(
        1-\exp\left[-\phi^{-1}\left(
        \frac{c}{t}\right)S_t\right]
        \right)\\
        &=1-\frac{1}{1-\eup^{-1}}
        \left(
        1-\exp\left[-t\phi\left(
        \phi^{-1}\left(
        \frac{c}{t}\right)
        \right)\right]\right)\\
        &=1-\frac{1-\eup^{-c}}{1-\eup^{-1}}\\
        &=\frac{\eup^{1-c}-1}{\eup-1}. \qedhere
    \end{align*}
\end{proof}

\subsection{Some estimates of the heat kernel}

As a transition density function of subordinate
Brownian motion, the heat kernel of $-\phi(-\Delta)$
is given by
    $$
        p_t(x-y)=\int_0^\infty
        \frac{1}{(4\pi s)^{1/2}}\,
        \eup^{-|x-y|^2/(4s)}\,\Pp(S_t\in\dup s),
    $$
where $S_t$ is the subordinator associated
with Bernstein function $\phi$.
Then it is easy to see the monotonicity:
\begin{equation}\label{monoto}
    p_t(x)\geq p_t(y)\quad \text{whenever $|x|\leq|y|$}.
\end{equation}

\begin{lemma}\label{factorlower}
    If $p_t(0)\leq1$ and $\tau\geq2$, then
    $$
        p_t\left(\frac{x-y}{\tau}\right)
        \geq p_t(x)p_t(y)
        \quad\text{for all $x,y\in\real^d$}.
    $$
\end{lemma}

\begin{proof}
    Noting that
    $$
        \frac{|x-y|}{\tau}\leq\frac{|x|+|y|}{\tau}
        \leq\frac{2(|x|\vee|y|)}{\tau}
        \leq|x|\vee|y|,
    $$
    it holds from the monotonicity \eqref{monoto}
    that for any $x,y\in\real^d$,
    $$
        p_t\left(\frac{x-y}{\tau}\right)\geq
        p_t(|x|\vee|y|)\geq p_t(|x|)\wedge
        p_t(|y|)\geq
        p_t(|x|)
        p_t(|y|),
    $$
    where in the last inequality we have used the
    fact that
    \[
        p_t(|x|)\vee
        p_t(|y|)\leq p_t(0)\leq1,
        \quad x,y\in\real^d. \qedhere
    \]
\end{proof}

\begin{lemma}\label{lower}
    There exists $c=c(d)>0$ such that for any $t>0$,
    $$
        p_t(x)\geq c
        \left[\phi^{-1}\left(
        \frac{1}{2t}\right)\right]^{d/2}
        \quad\text{provided}\quad
        |x|\leq\left[\phi^{-1}\left(
        \frac{2}{t}\right)\right]^{-1/2}.
    $$
\end{lemma}

\begin{proof}
    First, for any $t>0$ and $x\in\real^d$,
    \begin{align*}
        p_t(x)&\geq(4\pi)^{-d/2}
        \int_{\left[\phi^{-1}\left(
        \frac{2}{t}\right)\right]^{-1}}
        ^{\left[\phi^{-1}\left(
        \frac{1}{2t}\right)\right]^{-1}}
        s^{-d/2}\eup^{-|x|^2/(4s)}
        \,\Pp\left(S_t\in\dup s\right)\\
        &\geq(4\pi)^{-d/2}
        \left[\phi^{-1}\left(
        \frac{1}{2t}\right)\right]^{d/2}
        \eup^{-|x|^2\phi^{-1}\left(
        \frac{2}{t}\right)/4}
        \Pp\left(
        \left[\phi^{-1}\left(
        \frac{2}{t}\right)\right]^{-1}
        \leq S_t\leq
        \left[\phi^{-1}\left(
        \frac{1}{2t}\right)\right]^{-1}
        \right).
    \end{align*}
    If $|x|\leq\left[\phi^{-1}\left(
    \frac{2}{t}\right)\right]^{-1/2}$, it holds from
    Lemma \ref{probab} that
    \[
        p_t(x)\geq(4\pi)^{-d/2}\,
        \eup^{-1/4}\,
        \frac{\eup^{3/2}-2\eup+1}{\eup(\eup-1)}
        \left[\phi^{-1}\left(
        \frac{1}{2t}\right)\right]^{d/2}.  \qedhere
    \]
\end{proof}

\begin{lemma}
    There exists $c=c(d)>0$ such that for any $t>0$,
    $$
        \int_{|y|\leq\left[\phi^{-1}\left(
        \frac{1}{2t}\right)\right]^{-1/2}}
        p_t(y)\,\dup y\geq c.
    $$
\end{lemma}

\begin{proof}
    Applying Lemma \ref{subord} with $c=1/2$, it
    is easy to obtain that
    \begin{align*}
        &\int_{|y|\leq\left[\phi^{-1}\left(
        \frac{1}{2t}\right)\right]^{-1/2}}
        p_t(y)\,\dup y\\
        &\qquad\quad=(4\pi)^{-d/2}\int_0^\infty
        \left(
        \int_{|y|\leq\left[\phi^{-1}\left(
        \frac{1}{2t}\right)\right]^{-1/2}}
        \eup^{-|y|^2/(4s)}\,\dup y
        \right)s^{-d/2}\,\Pp\left(S_t\in\dup s\right)\\
        &\qquad\quad=\frac{1}{2^{d-1}\Gamma(d/2)}
        \int_0^\infty
        \left(
        \int_0^{\left[\phi^{-1}\left(
        \frac{1}{2t}\right)\right]^{-1/2}}r^{d-1}
        \eup^{-r^2/(4s)}\,\dup r
        \right)s^{-d/2}\,\Pp\left(S_t\in\dup s\right)\\
        &\qquad\quad=\frac{1}{\Gamma(d/2)}
        \int_0^\infty
        \left(
        \int_0^{\frac{1}{4s}\left[\phi^{-1}\left(
        \frac{1}{2t}\right)\right]^{-1}}r^{-1/2}
        \eup^{-r}\,\dup r
        \right)\,\Pp\left(S_t\in\dup s\right)\\
        &\qquad\quad=\frac{1}{\Gamma(d/2)}
        \int_0^\infty
        r^{-1/2}
        \eup^{-r}
        \Pp\left(
        S_t\leq\frac{1}{4r}\left[\phi^{-1}\left(
        \frac{1}{2t}\right)\right]^{-1}
        \right)
        \,\dup r\\
        &\qquad\quad\geq\frac{1}{\Gamma(d/2)}
        \Pp\left(
        S_t\leq\left[\phi^{-1}\left(
        \frac{1}{2t}\right)\right]^{-1}
        \right)
        \int_0^{1/4}
        r^{-1/2}
        \eup^{-r}
        \,\dup r\\
        &\qquad\quad\geq\frac{1}{\Gamma(d/2)}
        \frac{\sqrt{\eup}-1}{\eup-1}
        \int_0^{1/4}
        r^{-1/2}
        \,\dup r
        =\frac{1}{\Gamma(d/2)}
        \frac{\sqrt{\eup}-1}{\eup-1}.  \qedhere
    \end{align*}
\end{proof}

\begin{lemma}\label{heat}
    There exists $c=c(d)>0$ such that
    $$
        \int_{\real^d}p_t(y)^2\,\dup y\geq
        c\left[\phi^{-1}\left(
        \frac{1}{2t}\right)\right]^{d/2},
        \quad t>0.
    $$
\end{lemma}

\begin{proof}
    By the Jensen inequality and Tonelli's theorem,
    \begin{align*}
        &\left[\phi^{-1}\left(
        \frac{1}{2t}\right)\right]^{-d/2}
        \int_{\real}p_t(y)^2\,\dup y\\
        &\qquad\geq\left[\phi^{-1}\left(
        \frac{1}{2t}\right)\right]^{-d/2}(4\pi)^{-d}
        \int_{|y|\leq\left[\phi^{-1}\left(
        \frac{1}{2t}\right)\right]^{-1/2}}
        \left(\int_0^\infty
        s^{-d/2}\eup^{-|y|^2/(4s)}
        \,\Pp(S_t\in\dup s)
        \right)^2\,\dup y\\
        &\qquad\geq\frac{d\Gamma(d/2)}{2\pi^{d/2}}\,
        (4\pi)^{-d}
        \left(\int_0^\infty\left(
        \int_{|y|\leq\left[\phi^{-1}\left(
        \frac{1}{2t}\right)\right]^{-1/2}}
        \eup^{-|y|^2/(4s)}\,\dup y
        \right)s^{-d/2}\,\Pp(S_t\in\dup s)
        \right)^2\\
        &\qquad=\frac{2d}{4^d\pi^{d/2}\Gamma(d/2)}
        \left(\int_0^\infty\left(
        \int_0^{\left[\phi^{-1}\left(
        \frac{1}{2t}\right)\right]^{-1/2}}
        \eup^{-r^2/(4s)}r^{d-1}\,\dup r
        \right)s^{-d/2}\,\Pp(S_t\in\dup s)
        \right)^2\\
        &\qquad=\frac{d}{2\pi^{d/2}\Gamma(d/2)}
        \left(\int_0^\infty\left(
        \int_0^{\frac{1}{4s}\left[\phi^{-1}\left(
        \frac{1}{2t}\right)\right]^{-1}}
        \eup^{-z}z^{d/2-1}\,\dup z
        \right)\,\Pp(S_t\in\dup s)
        \right)^2\\
        &\qquad=\frac{d}{2\pi^{d/2}\Gamma(d/2)}
        \left(\int_0^\infty\eup^{-z}z^{d/2-1}
        \Pp\left(S_t\leq \frac{1}{4z}
        \left[\phi^{-1}\left(
        \frac{1}{2t}\right)\right]^{-1}\right)
        \,\dup z\right)^2\\
        &\qquad\geq\frac{d}{2\pi^{d/2}\Gamma(d/2)}
        \left(
        \Pp\left(S_t\leq
        \left[\phi^{-1}\left(
        \frac{1}{2t}\right)\right]^{-1}\right)
        \int_0^{1/4}\eup^{-z}z^{d/2-1}
        \,\dup z\right)^2\\
        &\qquad\geq\frac{d}{2\pi^{d/2}\Gamma(d/2)}
        \left(
        \frac{\sqrt{\eup}-1}{\eup-1}
        \int_0^{1/4}\eup^{-z}z^{d/2-1}
        \,\dup z\right)^2,
    \end{align*}
    where in the last inequality we have
    used Lemma \ref{subord} with
    $c=1/2$. This implies the desired lower bound.
\end{proof}

\begin{proof}[Alternate proof of Lemma \ref{heat}]
By Plancharel's theorem, we obtain
$$
\int_{\real^d}p_t(y)^2\,\dup y=\int_{\real^d}\hat{p}_t(\xi)^2\,
\dup \xi=\int_{\real^d}\eup^{-2t\phi(|\xi|^2)}\,\dup \xi
=\frac{2\pi^{d/2}}{\Gamma(d/2)}
\int_{0}^\infty\eup^{-2t\phi(r^2)}r^{d-1}\,\dup r.
$$
Since $\phi$ is an increasing function, this implies that
\begin{align*}
    \frac{\Gamma(d/2)}{2\pi^{d/2}}
    \int_{\real^d}p_t(y)^2\,\dup y
    &\geq\int_0^
    {\left[\phi^{-1}\left(
        \frac{1}{2t}\right)\right]^{1/2}}
        \eup^{-2t\phi(r^2)}r^{d-1}\,\dup r\\
    &\geq\exp\left[
        -2t\phi\circ\phi^{-1}\left(
        \frac{1}{2t}\right)
        \right]
        \int_0^
    {\left[\phi^{-1}\left(
        \frac{1}{2t}\right)\right]^{1/2}}
        r^{d-1}\,\dup r\\
    &=\frac{1}{\eup d}
    \left[\phi^{-1}\left(
        \frac{1}{2t}\right)\right]^{d/2},
\end{align*}
as required.
\end{proof}

\begin{proposition}\label{dw3d}
    Suppose that  Assumptions \ref{Assumption 3} holds.
    Then there exists
    $t_0>0$ and $c=c(d)>0$ such that
    $$
        (\mathcal{G}u)_{t+t_0}(x)\geq cK_{u_0}
        \quad\text{for all $t\in(0,t_0]$ and
        $x\in B(0,1)$}.
    $$
\end{proposition}

\begin{proof}
    Since
    $$
        p_t(0)=(4\pi)^{-d/2}\int_0^\infty s^{-d/2}
        \,\Pp\left(S_t\in\dup s\right)
        =(4\pi)^{-d/2}\E S_t^{-d/2},
    $$
    it follows from Assumption \ref{Assumption 3} and
    Lemma \ref{moment} that we can pick up
    $t_0\geq\frac{2}{\phi(1/4)}$ (and so, $\left[\phi^{-1}\left(
    2/t_0\right)\right]^{-1/2}\geq2$)
    such that $p_{t_0}(0)<1$. By Lemma \ref{factorlower},
    $$
        p_{t_0}(x-y)=p_{t_0}
        \left(\frac{2x-2y}{2}
        \right)
        \geq p_{t_0}(2x)p_{t_0}(2y),
    $$
    which, together with Lemma \ref{lower}, implies that
    \begin{align*}
        (\mathcal{G}u)_{t_0}(x)&\geq
        p_{t_0}(2x)\int_{\real^d}
        p_{t_0}(2y)u_0(y)\,\dup y\\
        &\geq p_{t_0}(2x)\int_{B(0,1)}
        p_{t_0}(2y)u_0(y)\,\dup y\\
        &\geq c
        \left[\phi^{-1}\left(
        \frac{1}{2t_0}\right)\right]^{d/2}
        p_{t_0}(2x)K_{u_0}
    \end{align*}
    for some constant $c=c(d)>0$.
    By the monotonicity \eqref{monoto} and
    the semigroup property,
    \begin{align*}
        (\mathcal{G}u)_{t_0+t}(x)
        &=\int_{\real^d}p_t(x-z)
        (\mathcal{G}u)_{t_0}(z)\,\dup z\\
        &\geq cK_{u_0}
        \left[\phi^{-1}\left(
        \frac{1}{2t_0}\right)\right]^{d/2}
        \int_{\real^d}p_t(x-z)
        p_{t_0}(2z)\,\dup z\\
        &\geq cK_{u_0}
        \left[\phi^{-1}\left(
        \frac{1}{2t_0}\right)\right]^{d/2}
        \int_{\real^d}p_t(2x-2z)
        p_{t_0}(2z)\,\dup z\\
        &=c_1K_{u_0}2^{-d}
        \left[\phi^{-1}\left(
        \frac{1}{2t_0}\right)\right]^{d/2}
        \int_{\real^d}p_t(2x-y)
        p_{t_0}(y)\,\dup y\\
        &=cK_{u_0}2^{-d}
        \left[\phi^{-1}\left(
        \frac{1}{2t_0}\right)\right]^{d/2}
        p_{t+t_0}(2x).
    \end{align*}
    For $x\in B(0,1)$ and $t\in(0,t_0]$,
    $$
        |2x|\leq2\leq\left[\phi^{-1}\left(
        \frac{2}{t_0}\right)\right]^{-1/2}
        \leq\left[\phi^{-1}\left(
        \frac{2}{t+t_0}\right)\right]^{-1/2},
    $$
    which, together with Lemma \ref{lower}, yields that
    $$
        p_{t+t_0}(2x)\geq c
        \left[\phi^{-1}\left(
        \frac{1}{2(t+t_0)}\right)\right]^{d/2}
        \geq c
        \left[\phi^{-1}\left(
        \frac{1}{4t_0}\right)\right]^{d/2}.
    $$
    Thus, for all $x\in B(0,1)$ and $t\leq t_0$,
    \[
        (\mathcal{G}u)_{t_0+t}(x)
        \geq c^2K_{u_0}2^{-d}
        \left[\phi^{-1}\left(
        \frac{1}{2t_0}\right)\phi^{-1}\left(
        \frac{1}{4t_0}\right)\right]^{d/2}. \qedhere
    \]
\end{proof}

Denote by $|\mathcal{A}|$ the Lebesgue measure of a
measurable subset $\mathcal{A}\subset\real^d$.

\begin{lemma}\label{volume}
    Let $0<r\leq R/2$, $x\in B(0,R)$, and $\mathcal{A}:=B(0,R)\cap B(x,r)$. Then there
    exist $c_i=c_i(d)>0$, $i=1,2$, such that
    $$
        c_1r^d\leq|\mathcal{A}|\leq c_2r^d.
    $$
\end{lemma}

\begin{proof}
    Since $\mathcal{A}\subset B(x,r)$, the upper bound
    is clear. It remains to prove the lower bound.
    If $|x|<R/2$, then it is easy to see that
    $\mathcal{A}=B(x,r)$ and the claim follows. If
    $R/2\leq |x|<R$, taking $\zeta$ on the line segment
    between $0$ and $x$ such
    that $|x-\zeta|=r/2$, then we have
    $B(\zeta,r/2)\subset\mathcal{A}$
    and thus $|\mathcal{A}|\geq|B(\zeta,r/2)|
    =c\,r^d$ for some $c=c(d)>0$; see Figure \ref{fig:illustration} for the illustration.
    \begin{figure}\label{fig:illustration}
    \begin{center}
        \includegraphics[width = .6\textwidth]{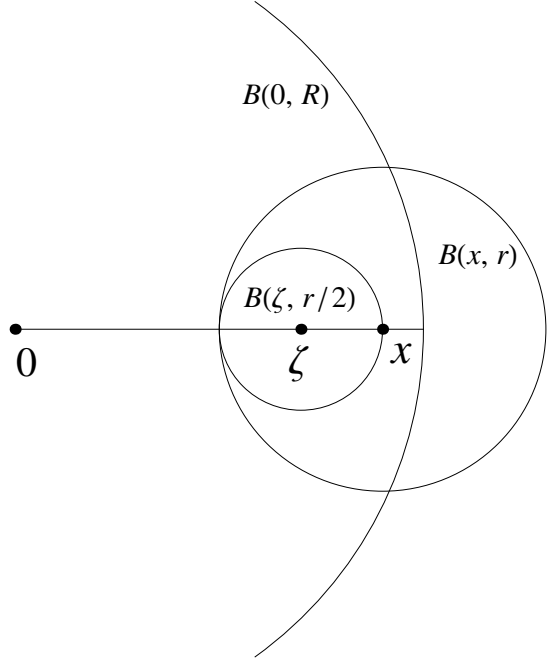}
    \end{center}
    \caption{Illustration of the proof of Lemma \ref{volume}}
    \end{figure}
\end{proof}

Let $\Phi:(0,\infty)\rightarrow(0,\infty)$ be
the strictly increasing function defined by
$$
    \Phi(r):=\frac{1}{\phi(r^{-2})},\quad r>0.
$$

\begin{proposition}\label{prop-kernel}
Suppose that Assumptions \ref{Assumption grow} and \ref{color} hold for some $R>0$.
Then there exists $c=c(d,R)>0$ such that
\begin{equation*}
\int_{B(0,\,R)\times B(0,\,R)}p_{t-s}(x_1-y_1)p_{t-s}(x_2-y_2)f(y_1,\,y_2)\,\d y_1\d y_2\geq c\,K_{R,f}
\end{equation*}
for all $x_1,\,x_2 \in B(0,\,R)$
and  $0\leq s\leq t\leq 2\Phi(R/2)$, where R is the same as that in Assumption \ref{color}.
\end{proposition}

\begin{proof}
First, Assumption \ref{color} gives
\begin{align*}
\int_{B(0,\,R)\times B(0,\,R)}&p_{t-s}(x_1-y_1)p_{t-s}(x_2-y_2)f(y_1,\,y_2)\,\d y_1\d y_2\\
&\geq K_{R,f}\int_{B(0,\,R)\times B(0,\,R)}p_{t-s}(x_1-y_1)p_{t-s}(x_2-y_2)\,\d y_1\d y_2.
\end{align*}
Let $0\leq s\leq t\leq 2\Phi(R/2)$. For $i=1,2$, set
\begin{equation*}
\mathcal{A}_i:=B(0,R)\cap B\Big(
x_i,\left[\phi^{-1}\left( 2/(t-s)\right)\right]^{-1/2}
\Big).
\end{equation*}
Noting that
  \begin{equation}\label{jh5fd}
      \frac{1}{2(t-s)}\geq\frac{1}{2t}
      \geq\frac{1}{4\Phi(R/2)}
      =\frac14\phi((2/R)^2),
  \end{equation}
one has
$$
    \left[\phi^{-1}\left( 2/(t-s)\right)\right]^{-1/2}
    \leq \frac R2.
$$
By Lemma \ref{volume}, there exists $c_1=c_1(d)>0$
such that
$$
|\mathcal{A}_i|\geq c_1
\left[\phi^{-1}\left( 2/(t-s)\right)\right]^{-d/2}.
$$
Then it follows from the heat kernel estimate
given in Lemma  \ref{lower} that
for some $c_2=c_2(d)>0$
\begin{align*}
\int_{B(0,\,R)\times B(0,\,R)}&p_{t-s}(x_1-y_1)p_{t-s}(x_2-y_2)\,\d y_1\d y_2\\
&\geq \int_{\mathcal{A}_1\times \mathcal{A}_2}p_{t-s}(x_1-y_1)p_{t-s}(x_2-y_2)\,\d y_1\d y_2\\
&\geq c_2\left[\phi^{-1}(1/(2(t-s)))\right]^d
|\mathcal{A}_1|\cdot|\mathcal{A}_2|\\
&\geq c_1^2c_2\left[\frac{\phi^{-1}(1/(2(t-s)))}
{\phi^{-1}(2/(t-s))}\right]^d.
\end{align*}
 By Lemma \ref{inverse}, there exists
  $c_3=c_3(R)>0$ such that
  $$
      \frac{\phi^{-1}(4r)}{\phi^{-1}(r)}\leq c_3
      \quad \text{for all $r\geq\frac14\phi((2/R)^2)$},
  $$
  which, together with \eqref{jh5fd}, implies
  $$
      \frac{\phi^{-1}(1/(2(t-s)))}{\phi^{-1}(2/(t-s))}
      \geq\frac{1}{c_3}.
  $$
 Combining all the estimates above,
 we get the required inequality.
\end{proof}

For an open subset $D\subset\real^d$, let
$p_{D,t}(x,y)$ denote the heat kernel of the
process (associated with the heat kernel $p_t(x,y)$)
killed upon exiting $D$.
According to Lemma \ref{equiva}, the following
lemma is a particular
case of \cite[Proposition 3.4]{kim-mimica-2018}.

\begin{lemma}\label{lowerbound}
    Suppose that $\phi$ is given by \eqref{bern}
    with $\phi(0+)=b=0$ and satisfies
    Assumption \ref{Assumption grow}.
    For any $R>0$, if
    Then
    for any $R>0$, there exists $c=c(d,R)>0$ such that
    $$
        p_{B(0,R),t}(x,y)\geq c\left[\Phi^{-1}(t)\right]^{-d}
    $$
    for all $0<t\leq\Phi(R/8)$ and $x,y\in B(0,R/2)$
    with $|x-y|\leq\Phi^{-1}(t)$.
\end{lemma}

For $R>0$, $t\geq0$ and $x\in\real^d$, let
\begin{equation*}
(\cG_{B(0,R)}u)_t(x):=\int_{B(0,\,R)}p_{D,t}(x,y)u_0(y)\,\d y.
\end{equation*}

\begin{proposition}\label{determ-dirich}
Suppose that $\phi$ is given by \eqref{bern}
    with $\phi(0+)=b=0$ and
satisfies Assumption \ref{Assumption grow}.
Let $R>0$ and $\varrho:=\inf_{x\in B(0,\,R/2)}u_0(x)>0$.
Then there exists $c=c(d,R)>0$ such that
for all $x\in B(0,\,R/2)$ and $t\leq\Phi(R/8)$,
$$
(\cG_{B(0,R)}u)_t(x)\geq c.
$$
\end{proposition}

\begin{proof}
The proof is quite straightforward. We use
Lemma \ref{lowerbound} to see
that for $t\leq \Phi(R/8)$,
\begin{align*}
(\cG_{B(0,R)}u)_t(x)&\geq \int_{B(0,\,R/2)}p_{B(0,R),t}
(x,y)u_0(y)\,\d y\\
&\geq\varrho
\int_{\{y\in B(0,\,R/2)\,:\, |x-y|\leq \Phi^{-1}(t) \}}p_{B(0,R),t}(x,y)\,\d y\\
&\geq\varrho c\big[\Phi^{-1}(t)\big]^{-d}
\cdot\big|
B(0,R/2)\cap B\big(x,\Phi^{-1}(t)\big)
\big|
\end{align*}
for some constant $c=c(d,R)>0$.
It remains to use Lemma \ref{volume}
to complete the proof.
\end{proof}

\begin{proposition}\label{prop-kernel-dirichlet}
Suppose
that Assumptions \ref{color} holds with
some $R>0$ and that $\phi$ is given by \eqref{bern}
    with $\phi(0+)=b=0$ and
satisfies Assumption \ref{Assumption grow}.
Then there exists
$c=c(d,R)>0$ such that
\begin{equation*}
\int_{B(0,\,R/2)\times B(0,\,R/2)}p_{B(0,R), t-s}(x_1-y_1)
p_{B(0,R), t-s}(x_2-y_2)f(y_1,\,y_2)\,
\d y_1\d y_2\geq cK_{R,f}
\end{equation*}
for all $x_1,\,x_2 \in B(0, R/2)$
and $0\leq s\leq t\leq \Phi(R/8)$.
\end{proposition}

\begin{proof}
Assumption \ref{color} gives
\begin{align*}
\int_{B(0,\,R/2)\times B(0,\,R/2)}&p_{B(0,R), t-s}(x_1-y_1)
p_{B(0,R), t-s}(x_2-y_2)f(y_1,\,y_2)\,\d y_1\d y_2\\
&\geq K_{R,f}\int_{B(0,\,R/2)\times B(0,\,R/2)}p_{D, t-s}(x_1-y_1)p_{B(0,R), t-s}(x_2-y_2)\,\d y_1\d y_2.
\end{align*}
Let $0\leq s\leq t\leq \Phi(R/8)$. We now use
Lemma \ref{volume} to
observe that if for $i=1,\,2$, we set
\begin{equation*}
\mathcal{A}_i:=B(0,R/2)\cap B\big(
x_i,\Phi^{-1} (t-s)
\big),
\end{equation*}
then $|\mathcal{A}_i|\geq c_1\left[\Phi^{-1}\left( (t-s)\right)\right]^{d}$ for some $c_1=c_1(d)>0$.
We therefore obtain from Lemma \ref{lowerbound}
that for some $c_2=c_2(R)>0$,
\begin{align*}
\int_{B(0,\,R/2)\times B(0,\,R/2)}&p_{B(0,R), t-s}(x_1-y_1)
p_{B(0,R), t-s}(x_2-y_2)\,\d y_1\d y_2\\
&\geq \int_{\mathcal{A}_1\times \mathcal{A}_2}
p_{B(0,R), t-s}(x_1-y_1)
p_{B(0,R),t-s}(x_2-y_2)\,\d y_1\d y_2\\
&\geq c_2^2\big[\Phi^{-1}(t-s)\big]^{-2d}\cdot
\big|\mathcal{A}_1\times \mathcal{A}_2\big|\\
&\geq c_2^2c_1^2.
\end{align*}
Combining all the estimates above, we have the required inequality.
\end{proof}

\subsection{Comparison results for nonlinear renewal type inequalities}

\begin{lemma}\label{explode}
    Suppose that Assumption \ref{Assumption 2} holds.
    Let $A$ and $B$ be two positive
    constants. If $g$ is a nonnegative function satisfying
    $$
        g(t)\geq A+B\left[\phi^{-1}\left(
        \frac{1}{2t}\right)\right]^{1/2}
        \int_0^tg(s)^{1+\varepsilon}\,\dup s,
        \quad t>0,
    $$
    where $\varepsilon>0$ is a constant, then there
    exists $t_0=t_0(A,B,\varepsilon,\delta)>1$ such
    that $g(t)=\infty$ for all $t\geq t_0$.
\end{lemma}

\begin{proof}
    Let
    $$
        h(t):=g(t)\left[\phi^{-1}\left(
        \frac{1}{2t}\right)\right]^{-1/2}.
    $$
    Since $\phi$ is increasing, for $t\geq1$,
    \begin{align*}
        h(t)&\geq A\left[\phi^{-1}\left(
        \frac{1}{2t}\right)\right]^{-1/2}
        +B\int_0^th(s)^{1+\varepsilon}
        \left[\phi^{-1}\left(
        \frac{1}{2s}\right)\right]^{(1+\varepsilon)/2}
        \,\dup s\\
        &\geq A\left[\phi^{-1}\left(
        \frac{1}{2}\right)\right]^{-1/2}
        +B\int_1^th(s)^{1+\varepsilon}
        \left[\phi^{-1}\left(
        \frac{1}{2s}\right)\right]^{(1+\varepsilon)/2}
        \,\dup s.
    \end{align*}
    Solving the ordinary differential equation
    $$
        \frac{f'(t)}{f(t)^{1+\varepsilon}}
        =B\left[\phi^{-1}\left(
        \frac{1}{2t}\right)\right]^{(1+\varepsilon)/2}
        ,\quad t\geq1
    $$
    with initial value
    $$
        f(1)=A\left[\phi^{-1}\left(
        \frac{1}{2}\right)\right]^{-1/2},
    $$
    we get
    $$
        f(t)=\left(
        A^{-\varepsilon}\left[\phi^{-1}\left(
        \frac{1}{2}\right)\right]^{\varepsilon/2}
        -B\varepsilon\int_1^t
        \left[\phi^{-1}\left(
        \frac{1}{2s}\right)\right]^{(1+\varepsilon)/2}
        \,\dup s
        \right)^{-1/\varepsilon},\quad t\geq1.
    $$
    Note that
    \begin{gather}\label{integral}
    \begin{aligned}
        \lim_{t\rightarrow\infty}\int_1^t
        \left[\phi^{-1}\left(
        \frac{1}{2s}\right)\right]^{(1+\varepsilon)/2}
        \,\dup s
        &=\int_1^\infty
        \left[\phi^{-1}\left(
        \frac{1}{2s}\right)\right]^{(1+\varepsilon)/2}
        \,\dup s\\
        &=\frac12\int_0^{\phi^{-1}\left(
        \frac{1}{2}\right)}
        r^{(1+\varepsilon)/2}\frac{\phi'(r)}{\phi(r)^2}
        \,\dup r\\
        &=\frac{1+\varepsilon}{4}\int_0^{\phi^{-1}\left(
        \frac{1}{2}\right)}\left(
        \int_0^rs^{(\varepsilon-1)/2}\,\dup s
        \right)\frac{\phi'(r)}{\phi(r)^2}
        \,\dup r\\
        &=\frac{1+\varepsilon}{4}\int_0^{\phi^{-1}\left(
        \frac{1}{2}\right)}\left(
        \int_s^{\phi^{-1}\left(
        \frac{1}{2}\right)}\frac{\phi'(r)}{\phi(r)^2}
        \,\dup r\right)s^{(\varepsilon-1)/2}\,\dup s\\
        &=\frac{1+\varepsilon}{4}\int_0^{\phi^{-1}\left(
        \frac{1}{2}\right)}\left(\frac{1}{\phi(s)}
        -2\right)s^{(\varepsilon-1)/2}\,\dup s\\
        &=\frac{1+\varepsilon}{4}\int_0^{\phi^{-1}\left(
        \frac{1}{2}\right)}\frac{\dup s}
        {s^{(1-\varepsilon)/2}\phi(s)}
        -\frac{1+\varepsilon}{2}\int_0^{\phi^{-1}\left(
        \frac{1}{2}\right)}s^{(\varepsilon-1)/2}
        \,\dup s\\
        &=\frac{1+\varepsilon}{4}\int_0^{\phi^{-1}\left(
        \frac{1}{2}\right)}\frac{\dup s}
        {s^{(1-\varepsilon)/2}\phi(s)}
        -\left[\phi^{-1}\left(
        \frac{1}{2}\right)\right]^{(1+\varepsilon)/2}.
    \end{aligned}
    \end{gather}

    \medskip

    \noindent
    \textbf{Case 1:} First, we consider the case that
    $\varepsilon\leq1-2\delta$. Then by Assumption \ref{Assumption 2}
    and \eqref{integral},
    \begin{align*}
        \lim_{t\rightarrow\infty}\int_1^t
        \left[\phi^{-1}\left(
        \frac{1}{2s}\right)\right]^{(1+\varepsilon)/2}
        \,\dup s
        &\geq\frac{1+\varepsilon}{4}\int_0^{1\wedge\phi^{-1}\left(
        \frac{1}{2}\right)}\frac{\dup s}
        {s^{(1-\varepsilon)/2}\phi(s)}
        -\left[\phi^{-1}\left(
        \frac{1}{2}\right)\right]^{(1+\varepsilon)/2}\\
        &\geq\frac{1+\varepsilon}{4}\int_0^{1\wedge\phi^{-1}\left(
        \frac{1}{2}\right)}\frac{\dup s}
        {s^{\delta}\phi(s)}
        -\left[\phi^{-1}\left(
        \frac{1}{2}\right)\right]^{(1+\varepsilon)/2}\\
        &=\infty.
    \end{align*}
    This implies that there exists
    $t_0=t_0(A,B,\varepsilon)>1$ such that
    $$
        B\varepsilon\int_1^{t_0}
        \left[\phi^{-1}\left(
        \frac{1}{2s}\right)\right]^{(1+\varepsilon)/2}
        \,\dup s
        =A^{-\varepsilon}\left[\phi^{-1}\left(
        \frac{1}{2}\right)\right]^{\varepsilon/2},
    $$
    and hence $f(t_0)=\infty$. By the comparison
    principle, we know that $g(t)=\infty$
    for all $t\geq t_0$.

    \bigskip

    \noindent
    \textbf{Case 2:} It remains to consider the case
    that $\varepsilon>1-2\delta$. Since $g(t)\geq A$
    for all $t>0$, it holds from the assumption that
    $$
        g(t)\geq A+BA^{\varepsilon-(1-2\delta)}
        \left[\phi^{-1}\left(
        \frac{1}{2t}\right)\right]^{1/2}
        \int_0^tg(s)^{1+(1-2\delta)}\,\dup s,
        \quad t>0.
    $$
    Applying the result in \textbf{Case 1} with $B$
    replaced by $BA^{\varepsilon-(1-2\delta)}$, and
    $\varepsilon$ replaced by $1-2\delta$,
    we conclude that there exists $t_0=t_0(A,B,\varepsilon,\delta)>1$ such that
    $g(t)=\infty$ for all $t\geq t_0$.
\end{proof}

%\newpage

\begin{lemma}\label{comparison}
    Let $T>0$ and  $g:(0,\infty)\rightarrow[0,\infty]$
    satisfy
    $$
        g(t)\geq A+B
        \left[\phi^{-1}\left(
        \frac{1}{2T}\right)\right]^d
        \int_0^tg(s)^{1+\gamma}
        \,\dup s
        \quad\text{for all $t\in(0,T]$},
    $$
    where $A,B,\gamma$ are positive numbers. Then
    for any $t_0\in(0,T]$, there exists $A_0=A_0(t_0)>0$
    such that for $A>A_0$,
    $$
        g(t)=\infty\quad \text{whenever $t\geq t_0$}.
    $$
\end{lemma}

\begin{proof}
    The solution to the integral equation
    $$
        f(t)=A+B
        \left[\phi^{-1}\left(
        \frac{1}{2T}\right)\right]^d
        \int_0^tf(s)^{1+\gamma}
        \,\dup s,\quad t\in[0,T],
    $$
    is given by
    $$
        f(t)=\left(
        A^{-\gamma}-\gamma B\left[\phi^{-1}\left(
        \frac{1}{2T}\right)\right]^dt
        \right)^{-1/\gamma}.
    $$
    The blowup occurs at $t=\gamma^{-1}A^{-\gamma}
    B^{-1}\left[\phi^{-1}\left(
        1/(2T)\right)\right]^{-d}$.
    For $t_0\in(0,T]$, take
    $A_0=\left[\gamma t_0B\right]^{-1/\gamma}
    \left[\phi^{-1}
    \left(1/(2T)
        \right)\right]^{-d/\gamma}$.
    If $A>A_0$, then we get blow-up for
    $f(t)$ before $t_0$. Finally, it remains
    to apply the comparison principle
    to finish the proof.
\end{proof}

\section{Proofs of the main results}

\begin{proof}[Proof of Theorem \ref{theorem-white}]
     By the Walsh isometry,
    Assumption \ref{Assumption 1}, and Jensen's inequality,
    \begin{align*}
        \E|u_t(x)|^2&=|(\mathcal{G}u)_{t}(x)|^2
        +\int_0^t\int_{\real}
        p_{t-s}(x-y)^2\E|\sigma(u_s(y))|^2
        \,\dup y\,\dup s\\
        &\geq \kappa^2+\int_0^t\left(
        \inf_{x\in\real}\E|u_s(x)|^2
        \right)^{1+\gamma}
        \left(\int_{\real}p_{t-s}(x-y)^2\,\dup y
        \right)\,\dup s.
    \end{align*}
    Setting
    $$
       F(t):=\inf_{x\in\real}\E|u_t(x)|^2,
    $$
    it follows from Lemma \ref{heat} (with $d=1$)
    that there exists
    a constant $c>0$ such that
    \begin{align*}
       F(t)&\geq \kappa^2+c\int_0^tF(s)^{1+\gamma}
       \left[\phi^{-1}\left(
        \frac{1}{2(t-s)}\right)\right]^{1/2}
        \,\dup s\\
        &\geq \kappa^2+c
        \left[\phi^{-1}\left(
        \frac{1}{2t}\right)\right]^{1/2}
        \int_0^tF(s)^{1+\gamma}
        \,\dup s.
    \end{align*}
    This, together with Lemma \ref{explode},
    implies the desired assertion.
\end{proof}

The following proposition is essential for the proof of Theorem \ref{theorem2-white}.

\begin{proposition}\label{local}
    Suppose that  Assumptions \ref{Assumption 1},
    \ref{Assumption 3} and \ref{color} hold. Then there exists
    $t_0>0$ and $K>0$ such that for
    all $t\geq t_0$
    $$
        \inf_{x\in B(0,1)}\E|u_t(x)|^2=\infty\quad
        \text{whenever $K_{u_0}\geq K$}.
    $$
\end{proposition}

\begin{proof}
    Let $t_0$ be as in Proposition \ref{dw3d}.
    By the Walsh isometry,
    \begin{align*}
        \E|u_{t+t_0}(x)|^2&=|(\mathcal{G}u)_{t+t_0}(x)|^2
        +\int_0^{t+t_0}\int_{\real^d}p_{t+t_0-s}(x-y)^2
        \E|\sigma(u_s(y))|^2\,\dup y\,\dup s\\
        &\geq|(\mathcal{G}u)_{t+t_0}(x)|^2
        +\int_0^{t}\int_{\real^d}p_{t-s}(x-y)^2
        \E|\sigma(u_{s+t_0}(y))|^2\,\dup y\,\dup s\\
        &=:I_1+I_2.
    \end{align*}
    For $x\in B(0,1)$ and $t\leq t_0$,
    by Proposition \ref{dw3d},
    $I_1\geq c_1K_{u_0}^2$ for some $c_1=c_1(d)>0$.
    By Assumption \ref{Assumption 1}  and the
    Jensen's inequality,
    \begin{align*}
        I_2&\geq\int_0^t\int_{B(0,1)}
        p_{t-s}(x-y)^2
        \left(\E|u_{s+t_0}(y)|^2\right)^{1+\gamma}
        \,\dup y\,\dup s\\
        &\geq \int_0^t
        \left(\inf_{x\in B(0,1)}\E|u_{s+t_0}(x)|^2\right)
        ^{1+\gamma}
        \int_{B(0,1)}
        p_{t-s}(x-y)^2
        \,\dup y\,\dup s.
    \end{align*}
    For $x,y\in B(0,1)$ and $0\leq s<t\leq t_0$,
    $$
        |x-y|\leq2\leq\left[\phi^{-1}\left(
        \frac{2}{t_0}\right)\right]^{-1/2}
        \leq\left[\phi^{-1}\left(
        \frac{2}{t-s}\right)\right]^{-1/2},
    $$
    which, together with Lemma \ref{lower},
    implies that for some $c_2=c_2(d)>0$,
    $$
        \int_{B(0,1)}p_{t-s}(x-y)^2\,\dup y
        \geq c_2\left[\phi^{-1}\left(
        \frac{1}{2(t-s)}\right)\right]^d
        \geq c_2\left[\phi^{-1}\left(
        \frac{1}{2t_0}\right)\right]^d.
    $$
    Setting
    $$
        G(t):=\inf_{x\in B(0,1)}\E|u_{t+t_0}(x)|^2,
    $$
    we find that
    $$
        G(t)\geq c_1K_{u_0}^2+c_2
        \left[\phi^{-1}\left(
        \frac{1}{2t_0}\right)\right]^d
        \int_0^tG(s)^{1+\gamma}
        \,\dup s
        \quad\text{for all $t\in(0,t_0]$}.
    $$
    Combining this with Lemma \ref{comparison},
    we complete the proof.
\end{proof}

\begin{proof}[Proof of Theorem \ref{theorem2-white}]
    Note that
    \begin{align*}
        \E|u_{t}(x)|^2&=|(\mathcal{G}u)_{t}(x)|^2
        +\int_0^{t}\int_{\real^d}p_{t-s}(x-y)^2
        \E|\sigma(u_s(y))|^2\,\dup y\,\dup s\\
        &\geq|(\mathcal{G}u)_{t}(x)|^2
        +\int_0^{t}\int_{\real^d}p_{t-s}(x-y)^2
        \left(\E|u_{s}(y)|^2\right)^{1+\gamma}
        \,\dup y\,\dup s\\
        &\geq |(\mathcal{G}u)_{t}(x)|^2
        +\int_{t_0}^{t}\int_{B(0,1)}p_{t-s}(x-y)^2
        \left(\E|u_{s}(y)|^2\right)^{1+\gamma}
        \,\dup y\,\dup s \\
        &{\geq |(\mathcal{G}u)_{t}(x)|^2 +\int_{t_0}^{t} \inf_{z \in B(0,1)} \left(\E|u_{s}(z)|^2\right)^{1+\gamma}\int_{B(0,1)}p_{t-s}(x-y)^2
         \,\dup y\,\dup s,}
    \end{align*}
    where $t_0$ is from Proposition \ref{local}.
    Now the desired claim follows immediately by
    Proposition  \ref{local}
    
    \iffalse
    \footnote{{\color{blue} Dear
    Erkan and Wei, I can not understand how you
    use Proposition \ref{local}.
    There is inf for $x\in B(0,1)$ in
    Proposition \ref{local}, but here we need
    all $x\in\real^d$. Please double check.} {\color{green}Since we look for $\geq$ here, the shrinkage of domain from $x\in\real^d$ to $x\in B(0,1)$ is fine. I added on an extra line to make it clearer.} }
    
    \fi.
\end{proof}

A new idea is required for the proof of Theorem \ref{theo2}. Briefly speaking, the non-linear renewal inequality on the second moment is not applicable here. We need to look at a different quantity, which is presented and proved as follows.
\begin{proposition}\label{prop-colored}
There exists a $t_0>0$ and $\kappa_0>0$ such
that whenever $\kappa:=\inf_{x\in\real^d}u_0(x)>\kappa_0$, then  for all $x,\,y\in B(0,\,R)$,

\iffalse\footnote{{\color{blue} What is $R$
in the next equation? The presentation of this proposition looks strange, please check.} \wei The R here is the same as that in Assumption \ref{color}. I changed those $B(0,\,1)$ in Proposition \ref{prop-colored} into $B(0,\,R)$ so that the presentation is in line with Assumption \ref{color}.}

\fi
\begin{align*}
\E|u_t(x)u_t(y)|=\infty\quad\text{for all}\quad  t\in \left( t_0,\, \left(\phi((2/R)^2)\right)^{-1} \right),
\end{align*}
where $R$ is the same as that in Assumption \ref{color}.
\end{proposition}
\begin{proof}
We start  with  the mild solution given in equation \eqref{mildcod} to obtain
\begin{align*}
\E|&u_t(x)u_t(y)|\\
&\geq \cG u_t(x)\cG u_t(y)+\int_0^t\int_{\real^d\times\real^d}
p_{t-s}(x-z)p_{t-s}(y-w)f(z,w)
\left(\E|u_s(z)u_s(w)|\right)^{1+\gamma}\,
\d z\,\d w\,\d s\\
&=:I_1+I_2.
\end{align*}
%We look at the term $I_1$ first.
The fact that the initial condition is bounded below by $\kappa$ gives
\begin{equation*}
I_1\geq \kappa^2.
\end{equation*}
We now assume that $t<\left(\phi((2/R)^2)\right)^{-1}$ and use Proposition \ref{prop-kernel} to obtain
\begin{align*}
I_2&\geq \int_0^t\left(\inf_{z,\,w\in B(0,\,R)}\E|u_s(z)u_s(w)|\right)^{1+\gamma}\int_{B(0,\,R)\times B(0,\,R)}p_{t-s}(x-z)p_{t-s}(y-w)f(z,w)\,
\d z\,\d w\,\d s\\
&\geq c_1K_{R,f}\int_0^t\left(\inf_{z,\,w\in B(0,\,R)}\E|u_s(z)u_s(w)|\right)^{1+\gamma}\, \d s.
\end{align*}
We now set
\begin{equation*}
G(s):=\inf_{x,\,y\in B(0,\,R)}\E|u_s(x)u_s(y)|.
\end{equation*}
We combine the estimates above to obtain
\begin{align*}
G(t)\geq \kappa^2+c_1K_{R,f}\int_0^tG(s)^{1+\gamma}\,\d s \quad\text{for}\quad t\leq\left(\phi((2/R)^2)\right)^{-1}.
\end{align*}
%The proof of Proposition \ref{volterra}
The blow-up is implied by  taking $\kappa$ big enough and we can make sure that $t_0$ is as small as we wish. The proof is completed.
\end{proof}

\begin{proof}[Proof of Theorem \ref{theo2}]
We use  the above proposition to prove the theorem. Indeed, from the mild formulation, we have
\begin{align*}
\E|u_t(x)|^2
\geq \kappa^2+\int_{t_0}^t\int_{B(0,\,R)\times B(0,\,R)}p_{t-s}(x-y)p_{t-s}(x-w)f(y,w)
\left(\E|u_s(y)u_s(w)|\right)^{1+\gamma}\,
\d y\,\d w\,\d s.
\end{align*}
The result now follows from the fact that all the functions involved on the right of the last inequality are strictly positive.
\end{proof}

}

We state the following proposition before the proof of Theorem \ref{energy-colored}.

{%\color{blue}
\begin{proposition}\label{prop5.2}
Let $u_t$ be the solution to \eqref{eq:colored}. Suppose that Assumption \ref{color} holds. Then there exists a $\tilde{t}>0$ such that for all $t\geq \tilde{t}$, we have
\begin{equation*}
\inf_{x,\,y\in B(0,\,R)}\E|u_t(x)u_t(y)|=\infty,
\end{equation*}
whenever $K_{u_0}>K$ for some positive constant $K$.
\end{proposition}

\begin{proof}
Our starting point is the mild formulation from which we obtain
\begin{align*}
\E|u_t(x)u_t(y)|
&\geq (\cG u)_t(x)(\cG u)_t(y)\\
&\quad+\int_0^t\int_{\real^d\times\real^d}
p_{t-s}(x-z)p_{t-s}(y-w)f(z-w)\E|\sigma(u_s(z))
\sigma(u_s(w))|\,\d z\, \d w\, \d s.
\end{align*}
The proof here essentially  follows the same idea as in the previous proofs. So, we  take $t_0$ as in Proposition \ref{dw3d} and set
\begin{equation*}
G(s):=\inf_{x,\,y\in B(0,\,R)}\E|u_{s+t_0}(x)u_{s+t_0}(y)|.
\end{equation*}
Using the ideas in Proposition \ref{prop-colored},
we obtain that
\begin{align*}
G(t)\geq c_1K_{u_0}^2+c_2K_{R,f}\int_0^tG(s)^{1+\gamma}\,\d s,
\end{align*}
for a suitable range of $t$. It is easy to see that this  finishes the proof.
\end{proof}

\begin{proof}[\bf Proof of Theorem \ref{energy-colored}]
With the proposition above, the proof of the theorem is now very similar to that of Theorem \ref{theorem2-white}  and is therefore omitted.
\end{proof}

The proof of Theorem \ref{Dirichlet} follows a similar pattern to the proofs of the previous results. We emphasize that in the case of \eqref{eq:dir}, the mild solution is given by
$$
u_t(x)=
(\cG_{B(0,R)} u)_t(x)+ \int_{\real^d}\int_0^t p_{B(0,R), t-s}(x-y)\sigma(u_s(y))\,F(\d s\,\d y).
$$

\begin{proof}[Proof of Theorem \ref{Dirichlet}]
As before, we have
\begin{align*}
&\E|u_t(x)u_t(y)|
\geq(\cG_{B(0,R)} u)_t(x)(\cG_{B(0,R)} u)_t(y)\\
&\qquad+\int_0^t\int_{B(0,\,R)\times B(0,\,R)}p_{B(0,R),t-s}(x-z)p_{B(0,R),t-s}(y-w)f(z-w)\left(\E|u_s(z))u_s(w)|\right)^{1+\gamma}\,\d z\,\d w\,\d s\\
&\quad=:I_1+I_2.
\end{align*}
%We look at $I_1$ first.
 By Proposition \ref{determ-dirich}, if $x,\,y \in B(0,\,R/2)$ and $t\leq\Phi(R/8)$, we have $I_1\geq c_1$
 for some $c_1=c_1(d,R)>0$.
 We now estimate  the second term $I_2$. Note that
\begin{align*}
I_2&\geq \int_0^t\left(\inf_{x,\,y\in B(0,\,R/2)}\E|u_s(x)u_s(y)| \right)^{1+\gamma}\\
&\quad\times \int_{B(0,\,R/2)\times B(0,\,R/2)}p_{B(0,\,R),t-s}(x-z)p_{B(0,\,R),t-s}(y-w)f(z-w)\,\d z\,\d w\,\d s.
\end{align*}
Combining this with Proposition \ref{prop-kernel-dirichlet},
we get that for $t\leq \Phi(R/8)$,
\begin{align*}
I_2&\geq c_2K_{R,f}\int_0^t\left(\inf_{x,\,y\in B(0,\,R/2)}\E|u_s(x)u_s(y)| \right)^{1+\gamma}\,\d s
\end{align*}
for some $c_2=c_2(d,R)>0$.
By combining the above inequalities and setting
\begin{equation*}
G(s):=\inf_{x,\,y\in B(0,\,R/2)}\E|u_s(x)u_s(y)|,
\end{equation*}
we have
\begin{align*}
G(t)\geq c_1
+c_2K_{R,f}\int_0^t G(s)^{1+\gamma}\,\d s,\quad t\leq\Phi(R/8).
\end{align*}
We see that for any $t_0<\Phi(R/8)$, there exists a  $\kappa_0$ such that if $\inf_{x\in B(0,\,R/2)}u_0(x)>\kappa_0$,
then $F(s)=\infty$ for all $s\geq t_0$. We finish the proof by the following observation
\begin{align*}
&\E|u_t(x)|^2 \geq |(\cG u)_t(x)|^2\\
&\;\,+\int_{t_0}^t\int_{B(0,\,R/2)\times B(0,\,R/2)}p_{B(0,\,R),t-s}(x-z)p_{B(0,\,R),t-s}(x-w)
f(z-w)\left(\E|(u_s(z))(u_s(y))|\right)^{1+\gamma}\,
\d z\,\d w\,\d s.
\end{align*}
By the positivity of all the relevant terms involved, we obtain the result using the fact that $F(s)=\infty$ for all $s\geq t_0$.
\end{proof}

\section{Conclusion and future research}\label{Sec:ConFR}
In this paper, we studied the non-existence of the global solutions to a class of stochastic partial  differential equations with a Bernstein function of the Laplacian operator acting on the  space variable. Both the cases of (space-time) white noise and  (space) colored noise were discussed. In each of the cases, we studied the blowup of the solution in $L^2$ sense with two different assumptions on the initial values. We also considered blowup for equations with  Dirichlet boundary conditions.
\par
By choosing certain forms of the Bernstein functions, the theorems in this paper cover several existing results. In addition, the general setting of the Bernstein function allows  our results to  include some equations that have not been studied yet.
\par
In  the future, we are  interested in the finite time blowup of stochastic partial  differential equations with some time fractional operators and fractional noise as well as equations  with a drift term. Another topic that we will investigate is the blowup of SPDEs in the almost sure sense.

\section*{Acknowledgement}
Wei Liu would like to thank the Natural Science Foundation of China (11701378, 11871343, 11971316), Chenguang Program supported by both Shanghai Education Development Foundation and Shanghai Municipal Education Commission (16CG50), and Shanghai Gaofeng \& Gaoyuan Project for University Academic Program Development for their supports.

\bigskip

%We are now ready to describe our findings in detail. For our first set of results, we will assume that the initial condition is bounded below by a positive constant. We therefore set
%\begin{equation}\label{minumum-values-u0}
%\inf_{x\in\real^d}u_0(x):=\kappa.
%\end{equation}

%\begin{theorem}\label{theo1}
%Let $u_t$ be be the solution to \eqref{eq:white} and suppose that $\kappa>0$. Then there exists a $t_0>0$ such that for all $x\in \R$,
%\begin{equation*}
%\E|u_t(x)|^2=\infty\quad\text{whenever}\quad t\geq t_0.
%\end{equation*}
%\end{theorem}

%\begin{assumption}\label{Assumption 4} The initial condition
%is nonnegative and satisfies
%$$
%    K_{u_0}:=\int_{B(0,1)}u_0(x)\,\dup x>0.
%$$

%\end{assumption}

\end{document}